\documentclass[reqno,11pt]{amsart}
\usepackage{amssymb,graphicx}
\usepackage{xcolor}
\usepackage[comma,numbers]{natbib}
\usepackage[mathlines,pagewise,left]{lineno}
\usepackage[notref,notcite,final]{showkeys}
\usepackage{multirow}
\usepackage[T1]{fontenc}
\usepackage{ae,aecompl}
\usepackage{hyperref}
\usepackage{bcol}
\usepackage{multicol}
\usepackage{caption}
\usepackage[colorinlistoftodos,prependcaption,textsize=small]{todonotes}

\definecolor{MyDarkBlue}{rgb}{0,0.08,0.45}
\newcommand{\REVTWO}[1]{{#1}}
\newcommand{\REV}[1]{{#1}}

\newcommand{\set}[1]{\ensuremath{ \left \{ #1 \right \}}}

\renewcommand{\P}{\ensuremath{\mathcal{ P }}}

\newcommand{\Q}{\ensuremath{\mathcal{ Q }}}

\newcommand{\R}{\ensuremath{\mathbb{R}}}

\usepackage{showkeys}

\newcommand{\func}{\rightarrow}

\newcommand{\Min}[1]{\ensuremath{\text{min} \left \{ #1 \right \}}}
\newcommand{\Max}[1]{\ensuremath{\text{max} \left \{ #1 \right \}}}

\newcommand{\0}{\ensuremath{\boldsymbol{0}}}
\newcommand{\1}{\ensuremath{\boldsymbol{1}}}

\newcommand{\RR}{\ensuremath{\mathbb{R}}}

\newcommand{\QQ}{\ensuremath{\mathbb{Q}}}

\DeclareMathOperator{\conv}{conv}

\newcommand{\defeq}{\ensuremath{:=}}

\newcommand{\be}{\begin{linenomath}\begin{equation}}
\newcommand{\benn}{\begin{linenomath}\begin{equation}\nonumber}
\newcommand{\ee}{\end{equation}\end{linenomath}}

\newcommand{\bml}{\begin{multline}}
\newcommand{\bmlnn}{\begin{linenomath}\begin{multline}\nonumber}
\newcommand{\eml}{\end{multline}}

\newcommand{\ba}{\begin{align}}
\newcommand{\ea}{\end{align}}

\newcommand{\sm}{\setminus}

\newcommand{\ind}[1]{\ensuremath{\chi_{#1}}}
\renewcommand{\sup}[1]{\ensuremath{S_{#1}}}

\newtheorem{proposition}{Proposition}

\theoremstyle{definition}
\newtheorem{definition}{Definition}

\newtheorem{example}{Example}

\theoremstyle{remark}
\newtheorem{remark}{Remark}

\newcommand{\polarP}{\ensuremath{\P_f^{\circ}}}

\begin{document}

\title[Submodular Function Minimization and Polarity]{Submodular function Minimization and Polarity}
\author{Alper Atamt{\"u}rk and Vishnu Narayanan}


\thanks{
A. Atamt\"urk: Department of Industrial Engineering and Operations
Research, University of California, Berkeley, CA 94720-1777 USA
\texttt{atamturk@berkeley.edu} \\
V. Narayanan: Industrial Engineering and Operations
Research, Indian Institute of Technology Bombay, Mumbai India 400076
\texttt{vishnu@iitb.ac.in}
}

\begin{linenomath}
\begin{abstract}
Using polarity, we give an outer polyhedral approximation for the epigraph of set functions.
For a submodular function, we prove that the corresponding polar relaxation is exact;
hence, it is equivalent to the Lov\'asz extension.
The polar approach provides an alternative proof for the convex hull description
of the epigraph of a submodular function. 
Computational experiments show that the inequalities from outer approximations can be effective as 
cutting planes for solving submodular as well as non-submodular set function minimization problems.


\vskip 3mm \normalsize \centering December 2019;  \REVTWO{November 2020}

\end{abstract}
\end{linenomath}

\ignore{
In this paper, we describe a polar relaxation of a set function. Whereas the Lov\'asz extension is
an inner piecewise polyhedral approximation of convex hull of the epigraph of a set function, the polar relaxation is
a convex polyhedral outer approximation.
For a submodular function, this outer polar relaxation is equal to the Lov\'asz extension, and thus 
provides an alternative proof of the convex hull description of the epigraph a
submodular function. 
Computational experiments on a number of specific submodular functions arising in queueing, risk management,
constrained min-cuts show the effectiveness of the approach.
}

\maketitle

\BCOLReport{19.02}

\setcounter{page}{1}


\vspace{-10mm}

\section{Introduction} \label{sec:intro}

Given a finite ground set $N$ and a rational valued set function $f: 2^N \rightarrow \QQ$, we
consider finding the minimum value of $f$:
\be \label{prob:submodmin}
\text{min}_{S \subseteq N} f(S).
\ee

Throughout, by abuse of notation,
we also use $f(x)$, where
$x \in \set{0,1}^N$ is the indicator vector for the set of subsets of $N$.
Introducing an auxiliary variable $z$ for the objective value, let us restate 
problem \eqref{prob:submodmin} as
\be
\min \set{z : (x,z) \in \Q_f}, \
\ee
where $\Q_f$ is the convex hull of the epigraph of $f$, i.e.,
\benn
\Q_f \defeq \conv \set{(x,z) \in \set{0,1}^N \times \RR: f(x) \le z} \cdot
\ee


In this paper, we give a polyhedral relaxation for $\Q_f$ based on polarity.
Whereas the Lov\'asz extension \citep{L:convex-ext} for a set function gives
an inner piecewise polyhedral approximation of the convex hull of its epigraph, the polar relaxation is
a (convex) polyhedral outer approximation. 

A set function $f: 2^N \rightarrow \QQ$ is \textit{submodular} if
	\benn
	f(S \cup T) + f(S \cap T) \le f(S) + f(T)  \ \text{for all} \ S, T \subseteq N.
	\ee
For a submodular set function, we show that the corresponding polar relaxation is exact; hence, it is
equivalent to the Lov\'asz extension. The polar approach, thus, provides a new polyhedral proof 
of the convex hull description of the epigraph of a submodular function.
Furthermore, through computational experiments, 
we show that the inequalities from outer approximations can be effective as 
cutting planes for solving submodular as well as non-submodular set function minimization problems.

\subsubsection*{A short literature review}

Submodular set functions play an important role in many fields and have received much interest in the literature \cite{E:submodular,EG:minmax-sub,topkis-minsub}.
Combinatorial optimization problems such as the min-cut problem, entropy minimization, matroids, 
binary quadratic function minimization with a non-positive matrix, among many others, are special cases.
Submodular functions can be minimized in polynomial time \citep{GLS:ellipsoid,iwata:2001,O:faster,S:combinatorial}. 
For comprehensive reviews on this subject
we refer the reader to \cite{F:submodularBook,I:submodular-survey,S:CObook}.
While the majority of research on submodularity \REV{is} devoted to optimization on binary variables,
submodularity has been useful for deriving strong inequalities for mixed 0--1 optimization as well
\cite{AG:m-matrix,AG:sub-cqmip,AJ:lifted-poly,AKT:pathcover,AN:conicobj,W:submodular-fc}. 
Unlike minimization, maximization of submodular functions is \NP-hard (the max-cut problem is a special case).
Most of the research for this case has been on approximation algorithms. The greedy algorithm and its extensions \cite{AG:max-util-poly,feige:maxsub,lee2010maximizing,NW:best-max,NWFM:maxsub} provide constant factor approximation. There is so far limited work 
on polyhedral analysis of submodular function maximization \cite{AA:max-util,LNG:maxsub,NW:submax-form,yu2017maximizing}.

\ignore{
	\REV{\citet{E:submodular} shows} that $\pi$ is an extreme point of $\P_f$ if and only if $\pi_i = f(S_{(i)}) - f(S_{(i-1)})$,
	where $S_{(i)} = \set{(1), (2), \ldots, (i)}$,  $1 \le i \le n$ for some
	permutation $((1), (2), \ldots, (n))$ of $N$.
	We refer to inequality $\pi x \le z$ defined by an extreme point $\pi$ of the extended polymatroid $\P_f$ as
	the \textit{extended polymatroid inequalities}.
	If $f$ is nondecreasing submodular, then $\P_f$ is called a polymatroid, in which case we refer to the inequalities simply as the \textit{polymatroid inequalities}.
}

\subsubsection*{Assumption and notation}
Without loss of generality, we assume that $f(\emptyset) = 0$ as,
otherwise, one can solve the equivalent minimization
problem for $f':= f - f(\emptyset)$, i.e.,
with $f'(S) = f(S) - f(\emptyset)$ for all $S \subseteq N$.
Let $\ind{S}$ be the indicator vector for a set $S \subseteq N$
and $\sup{x}$ be the support set of a binary vector $x \in \set{0,1}^N$.
For a vector $x \in \RR^N$ and $S \subseteq N$, define $x(S) := \sum_{i \in S} x_i$.

\section{Polar outer approximation}
\label{sec:polar}

\ignore{For submodular $f$, we refer to $\Q_f$ simply as the \textit{submodular
function polyhedron}.}

We start this section with a simple property of the facets of $\Q_f$. We refer to the variable bounds $\0 \le x \le \1$ as
the \textit{trivial inequalities} of $\Q_f$.

\begin{proposition} \label{prop:facet0}
Any non-trivial facet-defining inequality $\pi x \le  \alpha z + \pi_0$ for $\Q_f$
satisfies $\pi_0 \ge 0$ and $\alpha =1$ (up to scaling). 
\end{proposition}

\begin{proof}
Note that $\pi_0 \ge 0$ is necessary for validity, since otherwise
inequality $\pi x \le z + \pi_0$ is invalid for $(\0,0) \in \Q_f$. 
Because $(\0, 1)$ is a ray of $\Q_f$, inequality is invalid unless $\alpha \ge 0$. However, any valid
$\pi x \le \pi_0$ is implied by the trivial inequalities $\0 \le x \le \1$ as
$\set{(x,f(x)): x \in \set{0,1}^n } \subseteq \Q_f$. Thus, $\alpha > 0$ and, by scaling, it may be assumed to be one.
\end{proof}

For a set function $f$ with $f(\emptyset)=0$, let the \textit{associated polyhedron}\footnote{For the submodular case, $\P_f$ is referred to as the submodular polyhedron or the extended polymatroid.} be
\benn
\P_f := \set{\pi \in \RR^N: \pi(S) \le f(S) \text{ for all } S \subseteq N}.
\ee
Consider the polar of $\P_f$
\[
\polarP := \set{(x,z) \in \RR^N \times \RR: \pi x \le z, \text{ for all } \pi \in \P_f } \cdot
\]

The next proposition \REV{from  \citet{AN:conicobj} }
shows a polarity relationship between $\P_f$
and the homogeneous ($\pi_0 = 0$) valid inequalities for $\Q_f$. 

\begin{proposition} \REV{\cite{AN:conicobj}} \label{prop:polar}
Inequality $\pi x \le z$ is valid for $\Q_f$ if and only if $\pi \in P_f$.
\end{proposition}

\ignore{
\begin{proof}
	\REV{Let $x \in \set{0,1}^N.$}
For $\pi \in  \P_f$, we have $\pi x = \pi(\sup{x}) \le f(\sup{x}) \le z$. Conversely, if
$\pi  \not \in \P_f$, then $\pi(S) > f(S)$ for some $S \subseteq N$; but then
for $z = f(S)$, $\pi(S) = \pi \ind{S} > z$, contradicting the validity of $\pi x \le z$.
\end{proof}
}

We refer to inequalities of Proposition~\ref{prop:polar} as the
\textit{polar inequalities}.
By Proposition~\ref{prop:polar}, we have $\Q_f \subseteq \polarP$. 
Indeed, each facet of $\polarP$ is a facet of $\Q_f$ as well, as shown below.

\begin{proposition} \label{prop:facet}
Inequality $\pi x \le z$ is facet-defining for $\Q_f$ if and only if $\pi$ is an extreme point of $P_f$.
\end{proposition}

\begin{proof}
From Proposition~\ref{prop:polar} if $\pi \not \in \P_f$, inequality $\pi x \le z$ is invalid for $\Q_f$.
If $\pi \in \P_f$ is not an extreme point,
then $\pi = \lambda \pi^1 + (1-\lambda) \pi^2$ for some $0 < \lambda < 1$ and  distinct $\pi^1, \pi^2 \in \P_f$ and $\pi x \le z$
is implied by $\pi^1 x \le z$ and $\pi^2 x \le z$. Conversely, if $\pi$ is an extreme point of $P_f$,
it is the unique solution to a set of $n$ linearly independent equations $\pi(S_i) = f(S_i)$ for $i=1, \ldots, n$.
Then, the corresponding linearly independent points $(\ind{S_i}, f(S_i))$, $i=1, \ldots, n$ of $\Q_f$ and $(\0, 0)$
are on the face $\set{x \in \Q_f: \pi x = z}$. 
Finally, as $(\0, 1) \in \Q_f$ but not on the face, the face is proper.
\end{proof}

The polar relaxation $\polarP$ gives a (convex) polyhedral outer approximation for $\Q_f$. It indeed gives all nontrivial homogeneous facets of $\Q_f$.
 It is interesting to contrast it with the Lov\'asz extension \citep{L:convex-ext} for a set function.

\begin{definition}
	For a set function $f$, the \emph{Lov\'asz extension} $\hat f: [0,1]^N \rightarrow \RR$
	is defined as
	\[
	\hat f(x): = \sum_{i=1}^{n-1} (x_i-x_{i+1}) f(S_i) + x_n f(S_n),
	\]
	where $1 \ge x_1 \ge \ldots \ge x_n \ge 0$ and
	$S_i = \{1, 2, \ldots, i\}, \ i \in N =: \set{1, \ldots, n}$.  
\end{definition}
Observe that $\hat f$ is homogeneous and $\hat f(\chi_S) = f(S), \ \forall S \subseteq N$. It is easy to see that for a general set function, 
Lov\'asz extension is a piecewise polyhedral \textit{inner} approximation of $\Q_f$. Thus, we have
\[
\boxed{ \ \ \rule[-3mm]{0mm}{8mm}\textbf{epi } \hat f \ \subseteq \ \Q_f \ \subseteq \ \polarP \cap [0,1]^n \times \RR. \ \ }
\]

\begin{example} \label{ex:polar-lovas}
	In this example, we compare the polar outer approximation with the Lov\'asz extension for a non-submodular set function. 
	Consider
	the function $f$ defined as $f(\emptyset) = \ \ 0$ (A),
	$f(\set{1})=-1$  (B),
	$f(\set{2})=-1$ (C),
	$f(\set{1,2}) =-1$ (D). Function $f$ is, in fact, supermodular.
	\ignore{
		\begin{tabular}{lr}
			$f(\emptyset) $ & $= \ \ 0$ (A) \\
			$f(\set{1}) $ & $=-1$  (B) \\
			$f(\set{2}) $ & $=-1$ (C) \\
			$f(\set{1,2}) $ & $=-1$ (D)
		\end{tabular}
	}

\begin{multicols}{2}
	The inequalities describing \textbf{epi} $\hat f$, 	$\Q_f$, and $\polarP $,
	 other than the bounds $\0 \le x \le \1$, are listed below and displayed in Figure~1. \\
	
	\begin{tabular}{rl} 
		\textbf{epi} $\hat f$  : & $\left \{ \
\text{
	\begin{tabular}{lr}
	ABD: \ & $-x_1 \le z \ \text{ if }  x_1 \ge x_2$ \\
	ACD: \ & $-x_2 \le z \ \text{ if }  x_1 \le x_2$
	\end{tabular} 
}
\right .$ \vspace{2mm} \\ 
	$\Q_f$ : & $\left \{
	\text{
		\begin{tabular}{lr}
	\	ABC: \  & $-x_1-x_2 \le z$  \\
	\	BCD: \  & $-1 \le z$
		\end{tabular} 
	}
	\right .$ \vspace{2mm}
	\\ 
	$\polarP$ : &  $ \bigg \{
	\text{
		\begin{tabular}{lr}
	\	ABC: \  & $-x_1-x_2 \le z$  \\
		\end{tabular} 
	}
	$
	\end{tabular}

\columnbreak

    \begin{minipage}{\linewidth}
	\centering
		\begin{tikzpicture}[scale=1.4]
  \draw[thick](2,2,0)--(0,2,0)--(0,2,2)--(2,2,2)--(2,2,0)--(2,0,0)--(2,0,2)--(0,0,2)--(0,2,2);
  \draw[thick](2,2,2)--(2,0,2);
  \draw(2,0,0)--(0,0,0)--(0,2,0);
  \draw(0,0,0)--(0,0,2);
  \draw(0,-0.1,2.2) node{C};
  \draw(0,2.2,0) node{A};
  \draw(2.2,0,0) node{B};
  \draw(2.2,0,2.2) node{D};
  \draw(1, 2.2,0) node{ $x_1$ };
  \draw(0,2.25,1.5) node{ $x_2$ };
  \draw[dashed,gray](0,2,0)--(2,0,2);
  \draw[dashed,gray](0,2,0)--(2,0,0);
  \draw[dashed,gray](0,2,0)--(0,0,2);
  \draw[dashed,gray](0,0,2)--(2,0,0);
\end{tikzpicture}
	\captionof{figure}{	\textbf{epi} $\hat f$, $\Q_f$, $\polarP$.}
	\label{fig:cc}
\end{minipage}
\end{multicols}		
\noindent
Observe that 	\textbf{epi} $\hat f$ is non-convex and 
	\textbf{epi} $\hat f \subsetneq \Q_f \subsetneq \polarP \cap [0,1]^n \times \RR$.
	
\end{example}

\subsection*{The submodular case}

\citet{L:convex-ext} has shown that $\hat f$ is convex if and only if $f$ is submodular, establishing the relationship between convexity and submodular set functions. Next we show that for a submodular function $f$, all nontrivial facets of $\Q_f$ are homogeneous; consequently, the polar $\polarP$ gives an exact relaxation. So, for a submodular function $f$, it holds 
\[
\boxed{ \ \ \rule[-3mm]{0mm}{8mm}\textbf{epi } \hat f \ = \ \Q_f \ = \ \polarP \cap [0,1]^n \times \RR. \ \ }
\]

\begin{proposition} \label{prop:submodular}
For a submodular function $f$,
any non-trivial facet-defining inequality
\begin{equation}
    \pi x \leq z + \pi_0 \label{eq:positiveRHS}
\end{equation}
of $\Q_f$ satisfies \REV{$\pi_0 = 0$.}
\end{proposition}

\begin{proof}
\REV{Recall that $\pi_0 \ge 0$ by Proposition~\ref{prop:facet0}. 
	Therefore, it suffices to prove that $\pi \le 0$.}
For contradiction, suppose $\pi_0 > 0$.  Consider $f':2^N \func \R$ defined as $f'(\emptyset) = 0$ and $f'(S) = f(S) + \pi_0$ for $ \emptyset \neq S \subseteq N$. 
Observe that $f'$ is submodular as well.
Since~\eqref{eq:positiveRHS} is valid for $\Q_f$, $\pi(S) \leq f(S) + \pi_0$ for all $S \subseteq N$ and hence, $\pi \in \P_{f'}$ and,  by Proposition~2, inequality
\begin{equation}
    \pi x \leq z\label{eq:zeroRHS} 
\end{equation}
is valid for $\Q_{f'}$. Now,
since~\eqref{eq:positiveRHS} is facet-defining for $\Q_f$, 
~\eqref{eq:zeroRHS} is facet-defining for $\Q_{f'}$. Hence, by Proposition~3,
$\pi$ is an extreme point of \REV{polymatroid} $P_{f'}$. \REV{Due to \citet{E:submodular},}
after permuting variables, if necessary, we may assume that $\pi_1 = f'(S_1) = f(S_1) + \pi_0$, and $\pi_i = f'(S_i) - f'(S_{i-1}) = 
f(S_i) - f(S_{i-1})$ for $i =2,\dotsc, n$. 
Then, $\gamma \in \P_f$ for $\gamma_1 = \pi_1 - \pi_0$,
and $\gamma_i = \pi_i$ for all $i = 2,\dotsc, n$. However, by Proposition~2, $\gamma x \leq z$ is valid for $\Q_f$ and, together with $x_1 \le 1$, it dominates inequality~\eqref{eq:positiveRHS}.
\end{proof}

\begin{remark}
		Note that if $f(\emptyset) \neq 0$, in order to define the inequalities for $\Q_f$,
		we may use the polyhedron $\P_{f'}$ for $f' := f - f(\emptyset)$.
		In general, polar inequalities for $\Q_f$ take the form
		\be
		\pi x \le z - f(\emptyset), \ \pi \in \P_{f - f(\emptyset)}.
		\ee
		\ignore{as for $f'$ we have $\pi_i = f(S_{(i)}) - f(S_{(i-1)}) = f'(S_{(i)}) - f'(S_{(i-1)})$
			for any ordering of $N$. Hence, extended polymatroid inequalities are
			affine if and only if $f(\emptyset) \neq 0$.}
\end{remark}

\begin{remark}
For a submodular function,  $\P_f$ is an extended polymatroid and 
the polar inequalities reduce to the extended polymatroid inequalities.
The separation problem for extended polymatroid inequalities is
optimization of a linear objective over $\P_f$, which can be solved by
the greedy algorithm of Edmonds \cite{E:submodular}: Given $\bar x \in \RR_+^{n}$ and $\bar z \in \RR$, checking
\REV{whether $(\bar x, \bar z)$ violates an extended polymatroid inequality} is equivalent to
solving the problem
\be
\zeta := \text{max} \big \{\pi \bar x : \pi \in \P_f \big \}.
\ee
For a nonincreasing order
$\bar x_{(1)} \ge \bar x_{(2)} \ge \cdots \ge \bar x_{(n)}$, let
$S_{(i)} = \set{(1), (2), \ldots, (i)}$ and $\bar \pi_{(i)} = f(S_{(i)}) - f(S_{(i-1)})$
for $1 \le i \le n$. Then, $(\bar x, \bar z)$ is violated by
the corresponding extended polymatroid inequality if and only if $\zeta = \bar \pi \bar x > \bar z$.
\end{remark}

\REV{
\begin{remark}
\citet{AN:conicobj} study the epigraph of the risk function
$f(x) = a'x + h(c'x)$, where $a \in \RR^n, \ c \in \RR_+^n$ and $h$ is a univariate increasing concave function.
The convex hull proof in there is specific to $f$ and does not apply to a general submodular function.
\end{remark}
}

\REVTWO{
\begin{remark} Proposition~\ref{prop:submodular} shows that submodularity of $f$ implies $\Q_f = \polarP \cap [0,1]^n \times \RR$,
	which is true
	if and only if 
\begin{equation}
f(S) =  \Max{ \pi(S) \, : \, \pi \in \P_f} , \, \forall S \subseteq N. \label{eq:necessary}
\end{equation}
Note that \eqref{eq:necessary} does not hold for the non-submodular function $f$ in Example~\ref{ex:polar-lovas}: 
\[
f(\{1,2\}) = -1 > -2 = \Max{\pi_1 + \pi_2: \pi \in \P_f}.
\]
However, $\Q_f = \polarP \cap [0,1]^n \times \RR$ may hold also for non-submodular functions as Example~\ref{ex:nonsub} below illustrates.
Hence, the polar approach may be effective for non-submodular functions as well.
\end{remark}

\begin{example}\label{ex:nonsub}
For $N = \{1,2,3\}$ let $f \colon 2^N \to \RR$ be defined as $f(\emptyset) = 0$, $f(\{i\}) = 1$ for $i \in N$, $f(\{i,j\}) = 1.9$ for distinct $i,j\in N$, and $f(N) = 2.85$. 
Observe that $f$ is not submodular as 
\[
f(\{1,2\}) + f(\{1,3\}) = 3.8 < 3.85 = f(\{1,2,3\}) + f(\{1\})  \ .
\]
However, \eqref{eq:necessary} holds for this function: Note $\P_f = \{ \pi \in \RR^3 \ : \ \pi_i \leq 1 \text{ for } i \in N, \ \pi_i + \pi_j \leq 1.9 \text{ for distinct } i, j \in N, \ \pi_1 + \pi_2  + \pi_3 \leq 2.85 \}$.  Table~\ref{tab:optnonsub} lists the optimal solutions $\pi^*(S)$ and the optimal objective values for the linear program $g(S) := \Max{ \pi(S) \ : \ \pi \in \P_f}$ for all $S \subseteq N$.
\begin{table}[h!]
    \centering
    \caption{Verifying \eqref{eq:necessary} for Example~\ref{ex:nonsub}}.
    \label{tab:optnonsub}
    \begin{tabular}{c|c|c|c}
    \hline \hline
    $S$ & $f(S)$ & $\pi^*(S)$ & $g(S)$ \\
    \hline
    $\emptyset$ & 0 & $\pi^*_1 = \pi^*_2 = \pi^*_3 = 0$ & 0 \\
    $\{i\}, \ i \in N$ &1 & $\pi^*_i = 1$,  $\pi^*_j= 0$ for $j \in N \setminus \{i\}$ & 1 \\
    $\{i,j\}$, $i,j \in N$, $i \neq j$ & 1.9 & $\pi^*_i = 1, \pi^*_j =0.9$     & 1.9\\
    $\{1,2,3\}$ & 2.85 & $\pi^*_1 = \pi^*_2 = \pi^*_3 = 0.95$ & 2.85\\
    \hline \hline
    \end{tabular}
\end{table}
\end{example}
}

\ignore{
\subsection{Necessary conditions for $\Q_f = \polarP \cap [0,1]^n \times \RR$}
First, observe that $\Q_f = \polarP \cap [0,1]^n \times \RR$ if and only if 
\begin{equation}
	f(S) =  \Min{z \ : \  z \geq \pi^{\top}x, \pi \in \P_f, \ x \in [0,1]^n } =  \Max{ \pi(S) \ : \ \pi \in \P_f} , \ S \subseteq N , \label{eq:necessary}
\end{equation}
as the extreme points of $\Q_f$ are $(\chi_S,f(S))$ for all $S \subset N$.
d
Consider subsets $S, T \subseteq N$ with $S \cap T = \emptyset$. Then from~\eqref{eq:necessary}, 
\[
f(S \cup T) = \Max{ \pi(S) + \pi(T) \ : \ \pi \in \P_f } \leq f(S) + f(T) \ ,
\]
which is the submodularity condition when $S$ and $T$ are disjoint. However, it is not true in general that~\eqref{eq:necessary} implies that $f$ is submodular, as Example~\ref{ex:nonsub} below shows.

}

\section{Inequalities for general set functions}
\label{sec:general}

Unlike the submodular case, homogeneous inequalities of the polar relaxation are not 
sufficient to describe $\Q_f$ for a general set function.
One can, however, generate non-homogeneous inequalities by decomposing a set function into a sum of 
a submodular function and a supermodular function and utilizing inequalities for each.

It is well-known that for any set function $f$, such a decomposition exists. For a strictly submodular $h$ and
sufficiently large $\lambda > 0$, the first term below is submodular, whereas the second term is supermodular:
\begin{align*}
f =  \big (f + \lambda h \big ) + \big (-\lambda  h \big) \cdot
\end{align*}
Although, in general, a submodular--supermodular decomposition may be difficult to compute, in many cases, such a decomposition 
is either readily available or easy to construct.
Below we give a few examples that will also be used in the computations in Section~\ref{sec:comp}: \vskip 2mm
\begin{itemize}
	\setlength{\itemsep}{2mm}
	
	\item \textit{Optimization with higher moments}: When random variables deviate significantly from the normal distribution,
	optimization of utility functions with higher moments, such as skewness and kurtosis, in addition to the expectation and standard deviation, is preferred for more accurate modeling:  
	\[ \quad \quad 
	-\mu' x + \lambda_2 \left (\sum \sigma_i^2 x_i^2 \right )^{1/2} 
	- \lambda_3 \left (\sum \gamma_i^3 x_i^3 \right )^{1/3} 
	+ \lambda_4 \left (\sum \kappa_i^4 x_i^4 \right )^{1/4}
	\]
	Here, the first term is modular, the second and fourth (even) terms submodular, and the third (odd) term is supermodular. 
	
	\item \textit{Quadratic optimization on binaries}: Let $Q$ be a square matrix and $Q^-$ and $Q^+$ be the square matrices with the negative and positive elements of $Q$, respectively, and zeros elsewhere. Since $g(x) = x'Q^-x$ is submodular and $-h(x) = x'Q^+x$ is supermodular \cite{NW:IPbook},
	we have the corresponding submodular--supermodular decomposition: 
	\[ f(x) =	x'Qx = x'Q^-x + x'Q^+ x. \]

	\item \textit{Fractional linear functions}: Although not as immediate as in the cases above, 
	a submodular-supermodular decomposition can be
	constructed for a fractional linear function with positive coefficients $a,c >0$  as follows:
	\[
	f (x) := \frac{c'x}{1 + a'x} = \bigg (f(x) + \lambda \frac{a'x}{1 + a'x} \bigg ) - \bigg ( \lambda \frac{a'x}{1 + a'x} \bigg ) \cdot
	\]
	Observe that $h(x) = a'x/(1+a'x)$ is submodular as it is the composite of the univariate function $x/(1+x)$, which is concave over nonnegative values and the modular function $a'x$ \cite{AA:max-util}.
	\citet{HGP:assortment} show that $f(x)$ is submodular if $f(N) \le r_{min} := \min_{i \in N} c_i/a_i$. 
	Then
	letting $\lambda \ge \lambda_{min} := \big (c(N) - {r}_{min} (1+a(N))\big )^+$
	ensures that the first term $f(x)+\lambda h(x)$ in the decomposition is always submodular: 
	Observe that if all ratios $c_i/a_i$ are equal, then
	$f(x)$ is submodular and $\lambda_{min} = 0$ is sufficient. Otherwise, $\lambda = \lambda_{min}$ implies 
	$f(N) + \lambda_{min} h(N) \le r_{min} + \lambda_{min}$. \\

\end{itemize}

Writing a general set function as the difference of two submodular functions, $f=g-h$, 
an outer approximation for $f$ can be formed by using the polar inequalities for $g$ and the
\textit{submodular inequalities} of Nemhauser and Wolsey \cite[pg 710]{NW:IPbook} for $h$:
\be \label{eq:nw1} w \le h(S) -
\sum_{i \in S} \rho_i(N \sm i) (1-x_i) + \sum_{i \in N \sm S}
\rho_i(S) x_i \ \text{ for all } S \subseteq N, \ee 
\be
\label{eq:nw2} w \le h(S) - \sum_{i \in S} \rho_i(S \sm i) (1-x_i) +
\sum_{i \in N \sm S} \rho_i(\emptyset) x_i \ \text{ for all } S
\subseteq N, \ee
where $ \rho_i(S) = h(S \cup \set{i}) - h(S)$. Submodular inequalities 
\eqref{eq:nw1}--\eqref{eq:nw2} are valid for the \textit{hypograph} of $h$,
\textbf{hyp} $h := \set{(x,w) \in \set{0,1}^N \times \RR:  h(x) \ge w}\cdot$

\section{Computations} \label{sec:comp}

In this section we report on our computational experiments with using the
inequalities from submodular--supermodular decompositions for non-submodular functions. 
In each experiment, we control the deviation 
of the set functions from submodularity. All computations are done with Gurobi version 9.0 with default solver
options (except heuristics and presolve are turned off and single thread is used) on a Xeon workstation.

The first set of experiments are on binary quadratic optimization of the form: 
$\min \big \{x'Qx + \Omega c'x: x \in \set{0,1}^n \big \}$. 
The data is generated following \citet{carter-qp} : $Q_{ij}$ is drawn from Uniform$[-100\lambda,100(1-\lambda)]$ for $i \neq j$ and $Q_{ii} = 0$;
$c_i$ is drawn from Uniform$[-100(1-\lambda)(n-1),100\lambda(n-1)]$. 
The parameter $\lambda \in [0,1]$ controls the distance from submodularity:
for $\lambda=1$, $Q$ is nonpositive and the quadratic function is submodular, whereas
for $\lambda=0$, $Q$ is nonnegative and the quadratic function is supermodular. 
Thus, the quadratic function is a convex combination of a submodular function and 
a supermodular function. 
The parameter
$\Omega$ is chosen to ensure that quadratic and linear terms are well-balanced to avoid trivial solutions.
Gurobi 9.0 uses McCormick inequalities to automatically build convex relaxations of nonconvex quadratic functions.

\begin{table}[t]
	\caption{Binary quadratic optimization.}
	\label{tab:quad}
	\small
	\setlength{\tabcolsep}{6pt}
\begin{tabular}{c|rrrrrrrrrrr}
	\hline \hline
	$\lambda$	&	0.0	&	0.2	&	0.4	&	0.6	&	0.8	&	1.0	\\ \hline	
	gap (\%)	&	0.1	&	0.8	&	1.4	&	7.1	&	35.2	&	40.6	\\	
	cgap (\%)	&	0.1	&	0.6	&	0.2	&	0.1	&	0.1	&	0.0	\\	
	time (sec.)	&	0.2	&	2.4	&	2.7	&	2.4	&	5.6	&	8.6	\\	
	ctime (sec.)	&	0.2	&	1.8	&	3.1	&	1.4	&	3.0	&	0.0	\\	
	\# nodes	&	286.2	&	258.0	&	259.0	&	258.0	&	258.2	&	257.0	\\	
	\# cnodes	&	286.2	&	360.0	&	348.2	&	210.2	&	207.2	&	0.2	\\	
	\# cuts	&	0.0	&	1.6	&	4.0	&	6.8	&	13.0	&	2.0	\\ \hline \hline			
\end{tabular}
\end{table}
	
\ignore{	
$\lambda$	&	0.0	&	0.2	&	0.4	&	0.6	&	0.8	&	1.0	\\ \hline
gap (\%)	&	0.3	&	0.5	&	2.3	&	7.3	&	37.9	&	41.6	\\
cgap (\%)	&	0.3	&	0.3	&	0.3	&	0.2	&	0.3	&	0.0	\\
time (sec.)	&	1.0	&	0.5	&	0.5	&	0.4	&	0.5	&	0.6	\\
ctime (sec.)	&	0.8	&	0.3	&	0.4	&	0.2	&	1.0	&	0.0	\\
\# nodes	&	480.9	&	483.1	&	295.3	&	283.0	&	258.0	&	257.0	\\
\# cnode	&	480.9	&	541.8	&	318.9	&	89.9	&	106.2	&	0.0	\\
\# cuts	&	0.0	&	1.7	&	2.8	&	4.0	&	16.8	&	3.0	\\ \hline \hline

$\lambda$	&	0.0	&	0.2	&	0.4	&	0.6	&	0.8	&	1.0	\\ \hline
gap	(\%) &	0.7	&	1.3	&	1.8	&	20.5	&	19.0	&	44.3	\\
cgap (\%)	&	0.7	&	0.9	&	0.2	&	2.9	&	0.1	&	0.0	\\
time (sec.)	&	0.2	&	0.2	&	0.1	&	0.2	&	0.1	&	0.1	\\
ctime (sec.)	&	0.2	&	0.2	&	0.0	&	0.2	&	0.0	&	0.0	\\
\#nodes	&	797.4	&	881.2	&	297.6	&	258.6	&	257.0	&	257.0	\\
\# cnodes	&	797.4	&	641.4	&	140.0	&	156.8	&	30.4	&	5.0	\\
\# cuts	&	0.0	&	1.4	&	2.2	&	10.2	&	5.0	&	3.0	\\
\hline \hline
}

In Table~\ref{tab:quad} we report the integrality gap at the root node, solution time, and the number of nodes explored with and without
adding cuts. Each row shows the average for five instances with $n=200$.
We observe in Table~\ref{tab:quad} that the integrality gap of the convex relaxations increases with $\lambda$, 
achieving the highest value for the
submodular case ($\lambda=1$). The inequalities from the outer relaxations are particularly effective for these cases with high integrality gap. Indeed, for the submodular case, the gap is closed completely at the root node and the problems are solved without branching, as expected. For all values of $\lambda$ we observe a substantial reduction in the integrality gaps, leading to reduction in the number of nodes as well as the solution times.

The second set of experiments are done on the 0--1 knapsack problem with a mean-risk objective involving higher moments:  
\begin{align*}
\min & -\Omega \mu' x + \lambda  \left (\sum \sigma_i^2 x_i^2 \right )^{1/2} - (1-\lambda) \left (\sum \gamma_i^3 x_i^3 \right )^{1/3} + \lambda \left (\sum \kappa_i^4 x_i^4 \right )^{1/4}
\\ \text{s.t.} & \ a'x \le b, \ x \in \set{0,1}^n.
\end{align*}
The data is generated following \citet{BG:discrete}. The parameters $\mu_i$ and $a_i$ are drawn from Uniform[0,100], $\sigma_i$, 
$\gamma_i$, $\kappa_i$ are drawn from Uniform[0, $\mu_i$]. The knapsack capacity is set to $0.5 \sum_i a_i$.
As before, the parameter $\lambda$ controls the distance of the objective function 
from submodularity. For $\lambda = 1$ the objective is submodular; for $\lambda = 0$ it is supermodular. 
$\Omega$ is chosen to ensure that positive and negative terms in the objective are well-balanced to avoid trivial solutions.
Conic quadratic formulations of the epigraphs of (convex) standard deviation and kurtosis functions are standard \cite{NN:ConvexBook}.
We utilize the following convex formulation for the hypograph of the skewness function over binary $x$:
\[
s \le \left (\sum \gamma_i^3 x_i^3 \right )^{1/3}  \Longleftrightarrow z \le \sum_i \gamma_i^3 x_i; \ w^2 \le zs; \ s^2 \le w .
\]

The results for this experiment are summarized in Table~\ref{tab:moment}.
 Each row shows the average for five instances with $n=100$.
For this problem the integrality gap is highest for $\lambda =0.4$. Out of 30, 13 instances could not be solved to optimality within the half hour time limit without cuts. 
Of those, ten are solved to optimality within time limit with the cuts. Utilizing the cuts reduced the root gaps substantially and, consequently, led to smaller search trees, faster solution times and better solutions.

\begin{table}[t]
		\caption{Optimization with higher moments.}
	\label{tab:moment}
	\small
	\setlength{\tabcolsep}{6pt}
	\begin{tabular}{c|rrrrrrrrrrr}
		\hline \hline
		$\lambda$	&	0.0	&	0.2	&	0.4	&	0.6	&	0.8	&	1.0	\\ \hline
		gap (\%)	&	0.0	&	7.5	&	50.2	&	48.3	&	37.1	&	29.9	\\
		cgap (\%)	&	0.0	&	1.2	&	25.8	&	14.8	&	2.8	&	0.0	\\
		time (sec.)	&	0.1	&	71.9	&	1,800.0	&	1,553.0	&	1,279.3	&	388.1	\\
		ctime (sec.)	&	0.1	&	1.4	&	814.6	&	448.5	&	3.3	&	0.8	\\
		\# nodes	&	2.6	&	748.6	&	13,554.6	&	20,717.0	&	18,044.0	&	10,598.8	\\
		\# cnodes	&	2.6	&	57.6	&	3,720.2	&	4,107.4	&	87.8	&	2.2	\\
		\# cuts	&	0.0	&	52.6	&	100.0	&	100.0	&	95.4	&	64.2	\\ 
		\hline \hline
	\end{tabular}
\end{table}

Finally, the last set of experiments are done on fractional linear optimization problems of the form
\[
\min \bigg \{ \frac{c'x}{1 + a'x} - \Omega s'x  : \ x \in \set{0,1}^n \bigg \} \cdot
\]
The parameters $a_i$ and $s_i$ are drawn from Uniform[0,10] and $r_i$ are drawn from Uniform$[1+\lambda,2]$. 
We let $c_i = r_i a_i$, $i=1,\ldots,n$. 
$\Omega$ is chosen to ensure that positive and negative terms are well-balanced to avoid trivial solutions.
Observe that for $\lambda=1$, $c=2a$ and the objective is submodular. Otherwise, we form the submodular-supermodular decomposition as discussed in Section~\ref{sec:general}. The convex relaxation
used for the formulation is
\begin{align*}
\min & \ z + \lambda_{min} (t-1) - \Omega s'x \\
\text{s.t.} & \ w = 1+ a'x; \ \sum_i \tilde c_i x_i^2 \le zw; \ 1 \le tw; \ x, w,z,t \ge 0, 	
\end{align*}
where $\tilde c_i = c_i +\lambda_{min}a_i$, $i=1,\ldots,n$. Here $z$ corresponds to the convex relaxation of
the submodular function $\tilde c ' x / (1+a'x)$ and $t$ corresponds to the convex complement  $1/(1+a'x)$ of $a'x/(1+a'x)$. We utilize polymatroid inequalities for $\tilde c ' x / (1+a'x)$ and 
for the rotated cone constraint above \cite{AG:sub-cqmip}.

\begin{table}[t]
	\caption{Optimization with fractional linear functions.}
	\label{tab:frac}
	\small
	\setlength{\tabcolsep}{6pt}
	\begin{tabular}{c|rrrrrrrrrrr}
		\hline \hline
		$\lambda$	&	0.0	&	0.2	&	0.4	&	0.6	&	0.8	&	1.0	\\ \hline		
		gap (\%)	&	1,326.8	&	856.8	&	645.3	&	347.1	&	178.7	&	44.0	\\		
		cgap (\%)	&	90.8	&	61.5	&	41.3	&	21.5	&	9.6	&	0.0	\\		
		time (sec.)	&	83.3	&	117.2	&	206.6	&	261.4	&	84.5	&	40.8	\\		
		ctime (sec.)	&	44.1	&	88.9	&	69.5	&	71.1	&	12.3	&	0.0	\\		
		\# nodes	&	31,590.0	&	36,075.6	&	41,933.0	&	57,776.4	&	32,800.6	&	24,170.2	\\		
		\# cnodes	&	16,048.6	&	21,166.4	&	23,353.6	&	22,840.2	&	9,707.4	&	0.0	\\		
		\# cuts	&	27.6	&	27.8	&	30.4	&	23.2	&	23.4	&	22.0	\\ \hline \hline				
	\end{tabular}
\end{table}

Observe in Table~\ref{tab:frac} that the percentage integrality gaps are very large. This is due to the integer optimal values being close to zero. We see large reduction in integrality gaps when cuts ar added. 
For all instances the number of nodes and  the computation time
are reduced substantially with the largest improvement for the submodular case ($\lambda=1$) as expected. 

These computational experiments demonstrate that, when used as cutting planes, 
the valid inequalities from a submodular-supermodular decomposition of
general set functions can be effective to improve the branch-and-bound algorithms.

\ignore{
\section{Conclusion} \label{sec:conc}

We give a polar outer approximation for the epigraph for set functions. For submodular functions, polar relaxation yields the convex hull of the epigraph.  

}

\small
\bibliographystyle{apalike}
\bibliography{master}

\newcommand{\SortNoop}[1]{}
\begin{thebibliography}{}

\bibitem[Ahmed and Atamt{\"u}rk, 2011]{AA:max-util}
Ahmed, S. and Atamt{\"u}rk, A. (2011).
\newblock Maximizing a class of submodular utility functions.
\newblock {\em Mathematical Programming}, 128:149--169.

\bibitem[Atamt{\"u}rk and G\'omez, 2017]{AG:max-util-poly}
Atamt{\"u}rk, A. and G\'omez, A. (2017).
\newblock Maximizing a class of utility functions over the vertices of a
  polytope.
\newblock {\em Operations Research}, 65:433--445.

\bibitem[Atamt{\"u}rk and G{\'o}mez, 2018]{AG:m-matrix}
Atamt{\"u}rk, A. and G{\'o}mez, A. (2018).
\newblock Strong formulations for quadratic optimization with {M}-matrices and
  indicator variables.
\newblock {\em Mathematical Programming}, 170:141--176.

\bibitem[Atamt\"urk and G\'omez, 2020]{AG:sub-cqmip}
Atamt\"urk, A. and G\'omez, A. (2020).
\newblock Submodularity in conic quadratic mixed 0-1 optimization.
\newblock {\em Operations Research}, 68:609--630.

\bibitem[Atamt{\"u}rk and Jeon, 2019]{AJ:lifted-poly}
Atamt{\"u}rk, A. and Jeon, H. (2019).
\newblock Lifted polymatroid inequalities for mean-risk optimization with
  indicator variables.
\newblock {\em Journal of Global Optimization}, 73:677--699.

\bibitem[Atamt{\"u}rk et~al., 2017]{AKT:pathcover}
Atamt{\"u}rk, A., K{\"u}{\c{c}}{\"u}kyavuz, S., and Tezel, B. (2017).
\newblock Path cover and path pack inequalities for the capacitated
  fixed-charge network flow problem.
\newblock {\em SIAM Journal on Optimization}, 27:1943--1976.

\bibitem[Atamt{\"u}rk and Narayanan, 2008]{AN:conicobj}
Atamt{\"u}rk, A. and Narayanan, V. (2008).
\newblock Polymatroids and risk minimization in discrete optimization.
\newblock {\em Operations Research Letters}, 36:618--622.

\bibitem[Bergman and Cire, 2017]{BG:discrete}
Bergman, D. and Cire, A.~A. (2017).
\newblock Discrete nonlinear optimization by state-space decompositions.
\newblock {\em Management Science}, 64:4700--4720.

\bibitem[Carter, 1984]{carter-qp}
Carter, M.~W. (1984).
\newblock The indefinite zero-one quadratic problem.
\newblock {\em Discrete Applied Mathematics}, 7:23--44.

\bibitem[Edmonds, 1971]{E:submodular}
Edmonds, J. (1971).
\newblock Submodular functions, matroids and certain polyhedra.
\newblock In Guy, R., editor, {\em Combinatorial structres and their
  applications}, volume~11, pages 69--87. Gordon and Breach, New York, NY.

\bibitem[Edmonds and Giles, 1977]{EG:minmax-sub}
Edmonds, J. and Giles, R. (1977).
\newblock A min-max relation for submodular functions on graphs.
\newblock In Hammer, P., Johnson, E., Korte, B., and Nemhauser, G., editors,
  {\em Studies in Integer Programming}, volume~1 of {\em Annals of Discrete
  Mathematics}, pages 185 -- 204. Elsevier.

\bibitem[Feige et~al., 2011]{feige:maxsub}
Feige, U., Mirrokni, V.~S., and Vondr{\'a}k, J. (2011).
\newblock Maximizing non-monotone submodular functions.
\newblock {\em SIAM Journal on Computing}, 40:1133--1153.

\bibitem[Fujishige, 2005]{F:submodularBook}
Fujishige, S. (2005).
\newblock {\em Submodular Functions and Optimization}, volume~58 of {\em Annals
  of Discrete Mathematics}.
\newblock Elsevier, 2nd edition.

\bibitem[Gr{\"o}tschel et~al., 1981]{GLS:ellipsoid}
Gr{\"o}tschel, M., Lov{\'a}sz, L., and Schrijver, A. (1981).
\newblock The ellipsoid method and its consequences in combinatorial
  optimization.
\newblock {\em Combinatorica}, 1:169--197.

\bibitem[Han et~al., 2019]{HGP:assortment}
Han, S., G{\'o}mez, A., and Prokopyev, O.~A. (2019).
\newblock Assortment optimization and submodularity.

\bibitem[Iwata, 2008]{I:submodular-survey}
Iwata, S. (2008).
\newblock Submodular function minimization.
\newblock {\em Mathematical Programming}, 112:45--64.

\bibitem[Iwata et~al., 2001]{iwata:2001}
Iwata, S., Fleischer, L., and Fujishige, S. (2001).
\newblock A combinatorial strongly polynomial algorithm for minimizing
  submodular functions.
\newblock {\em Journal of the ACM}, 48:761--777.

\bibitem[Lee et~al., 1996]{LNG:maxsub}
Lee, H., Nemhauser, G.~L., and Wang, Y. (1996).
\newblock Maximizing a submodular function by integer programming: Polyhedral
  results for the quadratic case.
\newblock {\em European Journal of Operational Research}, 94:154--166.

\bibitem[Lee et~al., 2010]{lee2010maximizing}
Lee, J., Mirrokni, V.~S., Nagarajan, V., and Sviridenko, M. (2010).
\newblock Maximizing nonmonotone submodular functions under matroid or knapsack
  constraints.
\newblock {\em SIAM Journal on Discrete Mathematics}, 23:2053--2078.

\bibitem[Lov\'asz, 1983]{L:convex-ext}
Lov\'asz, L. (1983).
\newblock Submodular functions and convexity.
\newblock In Bachem, A., Gr\"otschel, M., and Korte, B., editors, {\em
  Mathematical Programming-- State of the Art}, pages 235--257. Springer,
  Berlin.

\bibitem[Nemhauser and Wolsey, 1981]{NW:submax-form}
Nemhauser, G. and Wolsey, L. (1981).
\newblock Maximizing submodular set functions: Formulations and analysis of
  algorithms.
\newblock In Hansen, P., editor, {\em Annals of Discrete Mathematics (11)},
  volume~59 of {\em North-Holland Mathematics Studies}, pages 279 -- 301.
  North-Holland.

\bibitem[Nemhauser and Wolsey, 1978]{NW:best-max}
Nemhauser, G.~L. and Wolsey, L.~A. (1978).
\newblock Best algorithms for approximating the maximum of a submodular set
  function.
\newblock {\em Mathematics of Operations Research}, 3:177--188.

\bibitem[Nemhauser and Wolsey, 1988]{NW:IPbook}
Nemhauser, G.~L. and Wolsey, L.~A. (1988).
\newblock {\em Integer and Combinatorial Optimization}.
\newblock John Wiley and Sons, New York.

\bibitem[Nemhauser et~al., 1978]{NWFM:maxsub}
Nemhauser, G.~L., Wolsey, L.~A., and Fisher, M.~L. (1978).
\newblock An analysis of approximations for maximizing submodular set
  functions{--I}.
\newblock {\em Mathematical Programming}, 14:265--294.

\bibitem[Nesterov and Nemirovski, 1993]{NN:ConvexBook}
Nesterov, Y. and Nemirovski, A. (1993).
\newblock {\em Interior-point polynomial algorithms for convex programming}.
\newblock SIAM, Philedelphia.

\bibitem[Orlin, 2009]{O:faster}
Orlin, J.~B. (2009).
\newblock A faster strongly polynomial time algorithm for submodular function
  minimization.
\newblock {\em Mathematical Programming}, 118:237--251.

\bibitem[Schrijver, 2000]{S:combinatorial}
Schrijver, A. (2000).
\newblock A combinatorial algorithm minimizing submodular functions in strongly
  polynomial time.
\newblock {\em Journal of Combinatorial Theory, Series B}, 80:346--355.

\bibitem[Schrijver, 2003]{S:CObook}
Schrijver, A. (2003).
\newblock {\em Combinatorial Optimization: Polyhedra and Efficiency}.
\newblock Springer Verlag, Berlin.

\bibitem[Topkis, 1978]{topkis-minsub}
Topkis, D.~M. (1978).
\newblock Minimizing a submodular function on a lattice.
\newblock {\em Operations Research}, 26:305--321.

\bibitem[Wolsey, 1988]{W:submodular-fc}
Wolsey, L.~A. (1988).
\newblock Submodularity and valid inequalities in capacitated fixed charge
  networks.
\newblock {\em Operations Research Letters}, 8:119--124.

\bibitem[Yu and Ahmed, 2017]{yu2017maximizing}
Yu, J. and Ahmed, S. (2017).
\newblock Maximizing a class of submodular utility functions with constraints.
\newblock {\em Mathematical Programming}, 162:145--164.

\end{thebibliography}

\end{document}